\documentclass[preprint,10t]{amsart}
\usepackage{amsmath}
\usepackage{amssymb,amsfonts,amsthm,latexsym,mathrsfs}
\usepackage{eucal}
\usepackage{color}
\usepackage{hyperref,bbm}
\input xypic
\theoremstyle{plain}

\newtheorem*{theoremA}{{\bf Theorem A}}

\newtheorem{corollary}[subsection]{{\bf Corollary}}
\newtheorem*{corollary*}{{\bf Corollary}}
\newtheorem{proposition}[subsection]{{\bf Proposition}}
\newtheorem{lemma}[subsection]{{\bf Lemma}}

\theoremstyle{definition}

\theoremstyle{remark}

\numberwithin{equation}{subsection}

\newcounter{ithmcount}

\makeatletter
\def\@author#1{\g@addto@macro\elsauthors{\normalsize%
    \def\baselinestretch{1}%
    \upshape\authorsep#1\unskip\textsuperscript{%
      \ifx\@fnmark\@empty\else\unskip\sep\@fnmark\let\sep=,\fi
      \ifx\@corref\@empty\else\unskip\sep\@corref\let\sep=,\fi
      }%
    \def\authorsep{\unskip,\space}%
    \global\let\@fnmark\@empty
    \global\let\@corref\@empty 
    \global\let\sep\@empty}%
    \@eadauthor={#1}
}
\makeatother

\begin{document}

\baselineskip=12.5pt
\title[]
{On finite groups in which the twisted conjugacy classes of the unit element are subgroups}

\author{Chiara Nicotera}
\address{Department of Mathematics, University of Salerno, Italy}
\email{cnicotera@unisa.it}
\thanks{The author has been supported by the National Group for Algebraic and Geometric Structures, and their Applications (GNSAGA - INdAM)}
\date{}

\renewcommand{\abstractname}{Abstract}
\begin{abstract}
\noindent
We consider  groups $G$ such that the set $[G,\varphi]=\{g^{-1}g^{\varphi}|g\in G\}$ is a subgroup for every automorphism $\varphi$ of $G$, and we prove that there exists such a group $G$ that is finite and nilpotent of class $n$ for every $n\in\mathbb N$. Then there exists an infinite nonnilpotent group with the above property and the conjecture 18.14 of $[5]$ is false.
\end{abstract}

\maketitle
\section{Introduction }
\label{i}
\noindent
Let $G$ be a group and $\varphi$ be an endomorphism of $G$; we say that elements $x,y\in G$ are $\varphi$-conjugate if there exists an element $z\in G$ such that $y=z^{-1}xz^{\varphi}$. 

It is easy to check that the relation of $\varphi$-conjugation is an equivalence relation in $G$. In particular it is the usual conjugation if $\varphi=id_G$ and it is the total equivalence relation if $\varphi =0_G$ is the zero endomorphism.

Equivalence classes are called twisted conjugacy classes or  $\varphi$-conjugacy classes and their number $R(\varphi)$  is called the Reidemeister number  of the endomorphism $\varphi$.

The $\varphi$-conjugacy class $[1]_{\varphi}=\{x^{-1}x^{\varphi}|x\in G\}$ of the unit element of the group $G$ is a subset whose cardinality is equal to the index $|G:C_G(\varphi)|$ of the centralizer $C_G(\varphi)=\{x\in G|x^{\varphi}=x\}$ of $\varphi$ in $G$; in what follows we will put $[1]_{\varphi}=:[G,\varphi]$.

If $\varphi=id_G$, then $[G,\varphi]=\{1\}$ and if $\varphi=0_G$, then  $[G,\varphi]=G$ so, in these cases, $[G,\varphi]$ is a subgroup of $G$. However, in the general case, $[G,\varphi]$ is not a subgroup, it is not if we consider an automorphism $\varphi\in AutG$ and not even if this automorphism $\varphi\in InnG$ is inner. For instance, if $G=S_3$ is the symmetric group of degree $3$ and $\varphi=\bar g$ is the inner automorphism induced by $g=(123)$, then $[1]_{\bar g}=\{1, (132)\}\not\le G.$

Notice that if $\varphi\in Aut_CG$ is a central automorphism of a group $G$, that is if $g^{-1}g^{\varphi}\in Z(G)$ for every $g\in G$, then $[G,\varphi]\le G$. So, in particular, if the group $G$ is nilpotent of class $\le 2$, then $[G,\varphi]$ is a subgroup for every $\varphi\in InnG$.

As it is easy to verify, if a group $G$ is abelian, then not only for every $\varphi\in EndG$, $H=[G,\varphi]\le G$, but for every element $x\in G$ its $\varphi$-cojugacy class is equal to the coset $[x]_{\varphi}=xH$, that is $\varphi$-conjugation is a congruence in $G$. 

It is  possible to prove that this property characterizes abelian groups. In fact if $\varphi$-conjugation is a congruence for every $\varphi\in EndG$, in particular conjugation is a congruence. This implies that every conjugacy class has order $1$ that is 
$g^x=g$ for every $g,x\in G$, hence $G$ is abelian. 
Actually for a group to be abelian, it is enough that there exists an inner automorphism $\bar g\in InnG$ such that $\bar g$-conjugation is a congruence in $G$.

\begin{proposition} A group $G$ is abelian if and only if there exists $g\in G$ such that $\bar g$-conjugacy is a congruence in $G$

\end{proposition}

\begin{proof} We have only to show that if $\bar g$-conjugacy is a congruence for an element $g\in G$, then the group is abelian.  

If $\bar g$ is a congruence, then $|[x]_{\bar g}|=|[1]_{\bar g}|=|\{h^{-1}h^g|h\in G\}|$ for every $x\in G$. In particular 
$|\{h^{-1}h^g=[h,g]|h\in G\}|=|[g]_{\bar g}|$ so $|\{[h,g]|h\in G\}|=1$ that is $\{[h,g]|h\in G\}=\{1\}$ since $[g]_{\bar g}=\{g\}$. This means that $g\in Z(G)$ hence $\bar g=id_G$, conjugacy is a congruence and so $G$ is abelian.

\end{proof}

In the paper $[1]$, the authors prove that is nilpotent every finite group $G$ in which the $\varphi$-conjugacy class of the unit element is a subgroup for every inner automorphism $\varphi\in InnG$.  Anyway it is not possible to bound the nilpotency class of such groups; in fact in $[3]$ the authors construct, for any integer $n>2$ and for any prime $p>2$ a finite $p$-group G, nilpotent of class $\ge n$ with this property. However they notice that there exist automorphisms $\phi\in AutG\setminus InnG$ such that $[G,\phi]$ is not a subgroup of $G$.

It is therefore natural to ask if it is possible to bound the nilpotency class of a finite group $G$ such that the $\varphi$-conjugacy class of the unit element is a subgroup for every $\varphi\in AutG$. In this regard, in $[1]$ the conjecture is made that such groups could be abelian (cfr. also $[5]$ 18.14). This conjecture is certainly false, in fact there exist finite non-abelian $p$-groups in which every automorphism is central (see for instance $[2]$,$[7]$). Of course such groups are nilpotent of class $2$ but we will prove the existence of nilpotent groups of every class $n$ with this property. So also the answer to the previous question is negative. 

Our main result is the following

\begin{theoremA}
 For every integer  $n\in \mathbb N$, there exists a  finite $p$-group $G$ nilpotent of class $n$ in which the $\varphi$-conjugacy class of the unit element is a subgroup for every $\varphi\in Aut G$. 
 \end{theoremA} 
 
 As we will see these groups are abelian-by-cyclic, and this result  will give  negative answer also to Problem 3 of $[3]$.
 
 Moreover from Theorem A, there follows the existence of an infinite non-nilpotent group $G$ such that $[G,\varphi]\le G$ for every $\varphi\in AutG$ and so conjecture 18.14 of [5] is false.    

\section {The proof of Theorem A.}

\noindent
Let $n$ be an integer $n\ge 2$, $p$ be an odd prime and $G=A\rtimes \langle x\rangle$, where $A=\langle a\rangle \times \langle b\rangle \simeq \mathbb Z_{p^n}\times \mathbb Z_{p^n}$, $\langle x\rangle \simeq \mathbb Z_{p^{n-1}}$ and $c^x=c^{1+p}$ for every $c\in A$.

It is easy to verify that $G'=\langle a^p\rangle \times \langle b^p\rangle$ and that $G$ is nilpotent of class $n$. Moreover, for every $c\in A$, we have that $\langle c\rangle$ is normal in $G$; in particular $\langle [g_1,g_2]\rangle$ is normal in $G$ for every
$g_1,g_2\in G$, $\langle g_1,g_2\rangle$ is a regular $p$-group for every $g_1,g_2\in G$ and so $G$ is a regular $p$-group.  ( $[4]$ III 10.2). 

Let $\varphi\in AutG$ be an automorphism of $G$, we will prove that $[1]_{\varphi}=\{g^{-1}g^{\varphi}|g\in G\}=[G,\varphi]$ is a subgroup.

First of all we have that 
$$(*)\>\>\>\> [c,\varphi]\in A\>\>\>\forall c\in A$$

In fact if there exists $c\in A$ such that $[c,\varphi]\not\in A$ then either $[a,\varphi]\not\in A$ or $[b,\varphi ]\not\in A$. W.L.O.G. we may assume that $[a,\varphi]\not\in A$ that is $[a,\varphi]=yx^{\alpha}$ with $y\in A$ and $\alpha\in\mathbb Z$ such that 
$p^{n-1}$ does not divide $ \alpha$.

So $a^{\varphi}=ayx^{\alpha}$, $(a^{\varphi})^{p^{n-1}}=(ay)^{p^{n-1}} (x^{\alpha})^{p^{n-1}}z^{p^{n-1}}$ with $z\in G'$, since $G$ is regular, and then $(a^{\varphi})^{p^{n-1}}=(ay)^{p^{n-1}}$ because $G'$ has exponent $p^{n-1}$ and, in particular, $ay$ has order
$p^n$.
Analogously $b^{\varphi}=btx^{\beta}$, where $t\in A$, $\beta\in \mathbb Z$ and $bt$ has order $p^{n}$.

Now $1=[a,b]=[a^{\varphi}, b^{\varphi}]=[ayx^{\alpha}, btx^{\beta}]=[ay, x^{\beta}]^{x^{\alpha}}[x^{\alpha}, bt]^{x^{\beta}}$ and this implies that $[ay, x^{\beta}]^{x^{\alpha}}=[bt, x^{\alpha}]^{x^{\beta}}$. Therefore there exist integer $\gamma, \delta \in\mathbb Z$ such that $(ay)^{\gamma}=(bt)^{\delta}$ because $[ay, x^{\beta}]^{x^{\alpha}}\in\langle ay\rangle$ and  $[bt, x^{\alpha}]^{x^{\beta}}\in \langle bt\rangle$.

From $\langle a\rangle\cap \langle b\rangle =\{1\}$ there follows that $\langle a^{\varphi}\rangle\cap \langle b^{\varphi}\rangle =\{1\}$, then $\langle (a^{\varphi})^{p^{n-1}}\rangle\cap \langle (b^{\varphi})^{p^{n-1}}\rangle =\langle (ay)^{p^{n-1}}\rangle\cap \langle (by)^{p^{n-1}}\rangle=\{1\}$, then $p^n$ divides both $\gamma$ and $\delta$ and so $x^{\beta}\in C_G(ay)$ and $x^{\alpha}\in C_G(bt)$. This implies that $x^{\alpha}, x^{\beta}\in A$, because $C_G(v)\subseteq A$ for every element $v\in A$ of order $p^n$.
Therefore $x^{\beta}=1=x^{\alpha}$, in particular $p^{n-1}$ divides $\alpha$ that is a contradiction. 

From $(*)$ there follows that $c^{\varphi}\in A$ for every $c\in A$, and in particular $a^\varphi\in A$. Now we prove that $x^{\varphi}=sx$ with $s\in A$ and so that $[x,\varphi]=x^{-1}x^{\varphi}=s^x=s^{1+p}\in A$.

Let $x^{\varphi}=sx^{\lambda}$ with $\lambda\in\mathbb Z$ and $s\in A$; from $a^x=a^{1+p}$ there follows that $(a^{\varphi})^{x^{\varphi}}=(a^{\varphi})^{(1+p)}$ that is $(a^{\varphi})^{x^{\lambda}}=(a^{\varphi})^{1+p}$ and 
$(a^{\varphi})^{(1+p)^{\lambda}}= (a^{\varphi})^{1+p}$. Therefore $(1+p)^{\lambda}\equiv (1+p)(mod p^n)$, $\lambda \equiv 1(mod p^n)$ and so $x^{\varphi}=sx$.

Since $[x,\varphi ] \in A$, for every $\alpha\in\mathbb Z$, we have that $[x^{\alpha},\varphi]=[x,\varphi]^{\beta}$ with $\beta\in\mathbb Z$ that is $[x^{\alpha},\varphi]\in \langle [x,\varphi]\rangle$ for every $\alpha \in \mathbb Z$ and so 
$V=\{[x^{\alpha},\varphi]|\alpha\in\mathbb Z\}\subseteq \langle [x,\varphi]\rangle$. 

We prove that conversely for every $\xi\in\mathbb Z$ there exists $\delta\in\mathbb Z$ such that $[x,\varphi]^{\xi}=[x^{\delta}, \varphi]$ and so $V=\langle [x,\varphi]\rangle$.

Put $[x,\varphi]=c$ that is $x^{\varphi}=xc$; we have that $c^{p^{n-1}}=1$ then the order $o(c)$ divides $p^{n-1}$. Now if $|V|=|\langle x\rangle:C_{\langle x\rangle}(\varphi)|=p^k$, then $(x^{\varphi})^{p^{k}}=(x^{p^{k}})^{\varphi}=x^{p^k}$ that is $(xc)^{p^{k}}=x^{p^{k}}$. So,  $x^{p^{k}}=x^{p^{k}}c^{p^{k}}z^{p^{k}}$ with $z\in (\langle xc\rangle)'=\langle c^p\rangle$ that is $x^{p^{k}}=x^{p^{k}}c^{p^{k}}(c^{lp})^{p^{k}}=x^{p^{k}}c^{(1+lp)p^{k}}$ that implies $c^{(1+lp)p^{k}}=1$ and so the order $o(c)=p^{n-1}$ divides $p^k$.

Then $p^k=|V|\le |\langle c\rangle|\le p^k$ and $V=\langle c\rangle$ since $|V|=|\langle c\rangle|$.

Let $B=\{[c,\varphi]|c\in A\}$, from $(*)$ we have that $B\subseteq A$ and we prove that $B\le A$. In fact, for every $c,d\in A$ we get $[cd,\varphi ]=[c,\varphi]^d[d,\varphi]=[c,\varphi][d,\varphi]$ and $[c^{-1},\varphi]=[c,\varphi]^{-1}$.
Moreover $\langle [x,\varphi]\rangle\le A$ since $[x,\varphi]\in A$. Then, $B\langle [x,\varphi]\rangle\le A$ and we show that $B\langle [x,\varphi]\rangle=[G,\varphi]$ so in particular $[G,\varphi]\le G$.

For every $g\in G$ there exist $f\in A$ and $\alpha \in \mathbb Z$ such that $g=fx^{\alpha}$; then $[g,\varphi]=[fx^{\alpha}, \varphi]=[f,\varphi]^{x^{\alpha}}[x^{\alpha}, \varphi]=[f,\varphi]^{\xi}[x,\varphi]^{\eta}$ with $\xi,\eta\in \mathbb Z$ and so $[g,\varphi ]\in B\langle [x,\varphi]\rangle$.
Conversely if we consider $[d,\varphi][x,\varphi]^{\eta}$ with $d\in A$, then we have that there exists $\xi\in\mathbb Z$ such that  $[d,\varphi][x,\varphi]^{\eta}= [d,\varphi][x^{\xi},\varphi]$ since there exists $\xi\in\mathbb Z$ such that $[x,\varphi]^{\eta}=[x^{\xi},\varphi]$. Let $\beta\in\mathbb Z$ such that $(1+p)^{\xi}\beta\equiv 1(mod p^n)$, then 
$[d^{\beta}x^{\xi},\varphi]=[d,\varphi]^{\beta x^{\xi}}[x^{\xi},\varphi]=[d, \varphi ]^{(1+p)^{\xi}\beta}[x^{\xi},\varphi]=[d,\varphi][x,\varphi]^{\eta}$ that is $[d,\varphi][x,\varphi]^{\eta}\in [G,\varphi]$ and this completes the proof.

\section { Further remarks and open questions}

\noindent
From Theorem A there follows that for every $n\in\mathbb N$ and for every odd prime $p$, there exists a $p$-group $P=G(n,p)$ nilpotent of class $n$ such that $[P,\varphi]\le P$ and so it is possible to construct an infinite non-nilpotent group with the same property.  

\begin{corollary}
There exists an infinite non-nilpotent group $G$ such that $[G,\varphi]\le G$ for every $\varphi\in Aut G$.
\end{corollary}

\begin{proof} For every $n\in\mathbb N$ fix a prime $p_n$ such that $p_n\not=p_m$ for every $m<n$ and put $P_n:=G(p_n,n)$; the group $G:=Dir_{n\in N}P_n$ is non-nilpotent and $[G,\varphi]\le G$ for every $\varphi\in AutG$. In fact, for every $\varphi\in G$ and for every $n\in \mathbb N$ we have that the restriction $\varphi _{/P_n}=:\varphi_{n}\in Aut P_n$, so $[P_n,\varphi_n]$ is a normal subgroup of $P_n$ and then it is a normal subgroup of $G$. We will show that $[G,\varphi]=\langle [P_n,\varphi_n]|n\in\mathbb N\rangle $ so, in particular, it is a subgroup.  

Let $g\in G$, then $g=g_{n_1}\dots g_{n_t}$ with $t, n_1,\dots n_t\in \mathbb N$ and $g_{n_i} \in P_{n_i}$ for every $i\in\{1,\dots t\}$.

So 
$[g,\varphi]= [g_{n_1}\dots g_{n_t},\varphi]= g_{n_t}^{-1}\dots g_{n_1}^{-1} (g_{n_1}\dots g_{n_t})^{\varphi}=
g_{n_1}^{-1}g_{n_1}^{\varphi_{n_1}}\dots g_{n_t}^{-1}g_{n_t}^{\varphi_{n_t}}=[g_{n_1},\varphi_{n_1}]\dots [g_{n_t},\varphi_{n_t}]$
and $[G,\varphi]\subseteq \langle [P_n,\varphi_n]|n\in \mathbb N\rangle$.

Conversely if $y\in  \langle [P_n,\varphi_n]|n\in \mathbb N\rangle$, then $y=y_{n_1}\dots y_{n_k}$ for some $k, n_1,\dots n_k\in \mathbb N$ and  $y_{n_i}=[z_{n_i}, \varphi_{n_i}]\in [P_{n_i}\varphi_{n_i}]$ and it is easy to see that
$y=[z_{n_1}\dots z_{n_t}, \varphi]\in [G,\varphi].$
\end{proof}

If $n=2$ Theorem A gives an example of abelian-by-(cyclic of order $p$)  $p$-group with the property for every odd prime $p$ that is a conterexamples to Problem 3 of $[3]$ for every odd prime.  

By Theorem A we can find  $p$-groups of class $n\ge 2$ with the property for odd primes only.  

What's about $2$-groups? There is a $2$-group of class $2$ with the property (see $[2]$), is it possible to have such a $2$ -group of class $>2$?

Notice that if we consider the $2$-group $G=A\rtimes \langle x\rangle$, with $A=\langle a\rangle \times \langle b\rangle\simeq \mathbb Z_{2^n}\times \mathbb Z_{2^n}$, $\langle x\rangle\simeq \mathbb Z_{2^{n-1}}$, $n>1$,  and $c^x=c^3$ for every $c\in A$, then there exists $\varphi\in AutG$ such that $[G,\varphi]\not\le G$.  

In order to define such an automorphism, notice that  for every $c\in A$ and for every $t\ge 1$, we have that $(xc)^{2^t}=x^{2^t}c^{2^{t+1}h}$ with $h\in\mathbb N$. For, if $t=1$, then $(xc)^2=xcxc=x^2c^xc=x^2c^4$; suppose that for some $t\in\mathbb N$ we have
$(xc)^{2^t}=x^{2^t}c^{2^{t+1}h}$, then $(xc)^{2^{t+1}}=((xc)^{2^{t}})^2=(x^{2^t}c^{2^{t+1}h})(x^{2^t}c^{2^{t+1}h})=x^{2^{t+1}}(c^{2^{t+1}h})^{x^{2^t}}(c^{2^{t+1}h})=x^{2^{t+1}}(c^{2^{t+1}h})^{3^{2^t}}c^{2^{t+1}h}=x^{2^{t+1}}c^{2^{t+1}h(3^{2^t}+1)}=x^{2^{t+1}}c^{2^{t+1}h(3^{2^t}+1)}= x^{2^{t+1}}c^{2^{t+2}k}$ since $3^{2^t}+1$ is even. 

Therefore if $c\in A$ has order $2^n$,we have that $xc$ has order $2^{n-1}$ and we may consider $\varphi\in Aut G$ such that $y^{\varphi}=y$ for every $y\in A$ and $x^{\varphi}=xc$. We have that $[x,\varphi]\in A$ and $[x^{\alpha},\varphi]\in A$ for every $\alpha\in \mathbb Z$. Therefore 
$[G,\varphi]=[\langle x\rangle, \varphi]$ because for every $g=x^{\alpha}y\in G$ with $\alpha\in Z$ and $y\in A$ we have that $[g,\varphi]=[x^{\alpha}y, \varphi]=[x^{\alpha}, \varphi]\in [\langle x\rangle, \varphi]$.
This implies that $|[G,\varphi]|=|[\langle x\rangle, \varphi]|=|\langle x\rangle :C_{\langle x\rangle}(\varphi)|\le 2^{n-1}$. Now $x^{-1}x^{\varphi}=c\in [G,\varphi]$, so $[G,\varphi]\not\le G$ since $c$ has order $2^{n}$.

Also every dihedral $2$-group $G=D_{2^n}=A\rtimes \langle x\rangle$ with $A=\langle a\rangle\simeq \mathbb Z_{2^{n-1}}$ and $\langle x\rangle \simeq \mathbb Z_2$ $(n>2)$, has an automorphism $\varphi\in Aut G$ such that $[G,\varphi]\not\le G$. 
In fact, if we consider the automorphism $\varphi$ defined by $a^\varphi=a$ and $x^{\varphi}=xa^{-1}$ then we have that $[G,\varphi]=[\langle x\rangle, \varphi]$ therefore $|[G,\varphi]| \le 2$ but $a^{-1}=x^{-1}x^{\varphi}\in [G,\varphi]$ then it is not a subgroup because $a^{-1}$ has order $2^{n-1}>2$.

Notice that, in particular if $n=3$ that is $G=D_8$, then $\varphi_{/Z(G)}=id_{Z(G)}$ because $[a,x]^{\varphi}=[a,xa^{-1}]=[a,x]$. Also the quaternion group $Q_8=\{1,-1,i,j,k,-i,-j,-k\}$ has an automorphism $\varphi$ such that $[Q_8,\varphi]\not\le Q_8$ and
$\varphi_{/Z(Q_8)}= id_{{Z(Q_8)}}$ actually the automorphism defined by $i^{\varphi}=i$ and $j^{\varphi}=k$, in this case $[Q_8,\varphi]=\{1,-i\}$.

Indeed it is possible to prove the following result:
\begin{proposition} For every prime $p$, if $G$ is an extraspecial $p$-group, then there exists $\varphi\in AutG$ such that $[G,\varphi]\not\le G$.
\end{proposition}

In order to show this proposition, first of all we prove the following lemmas

\begin{lemma} Suppose that a group $G$ is the central product of two subgroup $H$ and $K$; if there exists an automorphism $\phi\in Aut H$ such that $z^{\phi}=z$ for every $z\in Z(G)$ and $[H,\phi]\not\le H$, then there exists also $\varphi\in AutG$ such that 
$[G,\varphi]\not\le G$.
\end{lemma}

\begin{proof} The group $G$ is central product of $H$ and $K$ that is $G=HK$, $[H,K]=1$ and $H\cap K=Z(G)$. This implies that for every $g\in G$ we have that $g=hk$ with $h\in H$ and $k\in K$.
Let $\phi\in AutH$; if the restriction $\phi_{/Z(G)}=id_{/Z(G)}$, then the position $g^{\varphi}=(hk)^{\varphi}=h^{\phi}k$ defines a map $\varphi:G\rightarrow G$. In fact $hk=h_1k_1$ with $h,h_1\in H$ and $k,k_1\in K$ if and only if $h_1^{-1}h=k_1k^{-1}\in Z(G)$, so
$(h_1^{-1})^{\phi}h^{\phi}=(h_1^{-1}h)^{\phi}=(k_1k^{-1})^{\phi}=k_1k^{-1}$ that is $h^{\phi}k=h_1^{\phi}k_1$. It is easy to check that $\varphi\in AutG$, moreover $[G,\phi]=\{g^{-1}g^{\varphi}|g\in G\}=\{k^{-1}h^{-1}h^{\phi}k|h\in H, k\in K\}=\{h^{-1}h^{\phi}|h\in H\}=[H,\phi].$
\end{proof}

\begin{lemma} Let $p$ be a prime, and $G$ be a group of order $|G|=p^3$; if $G$ is non-abelian then there exists $\varphi\in AutG$ such that $z^{\varphi}=z$ for every $z\in Z(G)$ and  $[G,\varphi]\not\le G$.
\end{lemma}

\begin{proof} If $p=2$ then either $G\simeq D_8$ or $G\simeq Q_8$ and the claim is true. 

So we may suppose that $p$ is odd and in this case we have that either $G=\langle x,y,z|x^p=y^p=z^p=1, [x,z]=[y,z]=1,[x,y]=z\rangle$ that is $G=H\rtimes\langle y\rangle$ with $H=\langle x,z\rangle \simeq \mathbb Z_p\times\mathbb Z_p$, $\langle y\rangle\simeq \mathbb Z_p$ and $x^y=xz$, $z^y=z$, or $G=\langle x,y|x^{p^2}=y^p=1, x^y=x^{1+p}\rangle$ that is $G=\langle x\rangle \rtimes \langle y\rangle$ with $x^{p^2}=1$, $y^p=1$ and $x^y=x^{1+p}$.

In the first case we may consider $\varphi\in Aut G$ defined by $x^{\varphi}=x$ and $y^{\varphi}=yx$ and we have that $z^{\varphi}=[x,y]^{\varphi}=[x,yx]=[x,y]=z$ that is $\varphi_{/Z(G)}=id_{Z(G)}$. Moreover if we consider 
$g=y^tx^nz^m\in G$ with $0\le t,n,m\le p-1$, then we have that $g^{\varphi}=(yx)^tx^ny^m=y^tx^tz^{{t(t-1)\over 2}}x^ny^m= y^tx^{t+n}z^{{t(t-1)\over 2}+m}$ e quindi $g^{-1}g^{\varphi}=(y^{-t}x^{-n}z^{nt-m})(y^tx^{t+n}z^{{t(t-1)\over 2}+m})= x^tz^{{t(t-1)\over 2}}$ with $0\le t\le p-1$. 

Hence $[G,\varphi]=\{x^tz^{{t(t-1)\over 2}}|0\le t\le p-1\}$ and this set is not a subgroup of $G$.

\smallskip

In the second case we may consider the automorphism $\phi\in Aut G$ defined by  $\phi (y)=y$ e $\phi (x)=yx$; since $\phi([x,y])=[yx,y]=[x,y]$, we have that $\phi_{/Z(G)}=id_{Z(G)}$. Moreover if we consider $g=y^nx^m\in G$ with $0\le n\le p-1$ and $0\le m\le p^2-1$, then we have that $\phi(g)=y^{n+m}x^{m(1+p{(m-1)\over 2})}$.  Therefore $g^{-1}\phi(g)=y^mx^{m(1-(1+p)^m+p{(m-1)\over 2})}$ and the set of these elements, with  $0\le m\le p^2-1$  $0$ is not a subgroup of $G$.

\end{proof}

For every prime $p$, an extra-special $p$-group is the iterated central product of non-abelian groups of order $p^3$ (see for instance Lemma 2.2.9 of [6]), then from the two previous lemmas there follows Proposition 3.2

\end{document}